\documentclass{amsart}

\usepackage{amssymb}
\usepackage{amsmath}
\usepackage{amsthm}
\usepackage{amsfonts}

\newtheorem{Theorem}{Theorem}
\newtheorem{prop}[Theorem]{Proposition}
\newtheorem{cor}[Theorem]{Corollary}
\theoremstyle{definition}
\newtheorem*{remark}{Remark}

\newcommand\inprod[2]{B(#1,#2)}
\newcommand\bqf[4]{\left(\begin{array}{cc}#1&#2\\#3&#4\end{array}\right)}

\title{The 8-universality Criterion is Unique}
\author[Scott Duke Kominers]{Scott Duke Kominers}\thanks{The author was
  supported by a Harvard Mathematics Department Highbridge Fellowship.}
\address{\newline\indent Department of Mathematics, Harvard University\newline\indent c/o 8520 Burning Tree Road, Bethesda, MD 20817}\email{kominers@fas.harvard.edu}
\date{\today}

\begin{document}
\begin{abstract}Using the methods developed for the proof that the $2$-universality criterion is unique, we partially characterize criteria for the $n$-universality of positive-definite integer-matrix quadratic forms.  We then obtain the uniqueness of Oh's $8$-universality criterion as an application of our characterization results.\end{abstract}
\maketitle
\section{Introduction}
A degree-two homogenous polynomial in $n$ independent variables is called a \textit{quadratic form} (or just \emph{form}) \emph{of rank $n$}.   For a rank-$n$ quadratic form $Q(x_1,\ldots,x_n)= \sum_{i,j}a_{ij}x_ix_j$ (where $a_{ij}=a_{ji}$), the matrix given by $L=(a_{ij})$ is the \emph{Gram Matrix} of a $\mathbb{Z}$-lattice $L$ equipped with a  symmetric bilinear form $\inprod{\cdot}{\cdot}$ such that $\inprod{L}{L}\subseteq \mathbb{Z}$.  Then, $Q(\mathbf{x})=\mathbf{x}^TL\mathbf{x}=\inprod{L\mathbf{x}}{\mathbf{x}}$ for $\mathbf{x}\in\mathbb{R}^n$.

A rank-$n$ quadratic form $Q$ is said to \textit{represent} an integer $k$ if there exists an $\mathbf{x}\in\mathbb{Z}^n$ such that $Q(\mathbf{x})=k$.  More generally, a $\mathbb{Z}$-lattice $L$ \textit{represents} another $\mathbb{Z}$-lattice $\ell$ if there exists a $\mathbb{Z}$-linear, bilinear form-preserving injection $\ell \to L$.  A quadratic form is called \textit{universal} if it represents all positive integers.  Analogously, a lattice is called \textit{$n$-universal} if it represents all positive-definite integer-matrix rank-$n$ quadratic forms. Connecting these two notions of universality, we observe that a rank-$n$ quadratic form $Q$ is universal if and only if it is $1$-universal, as for an integer~$k$, $$k = Q(x_1,\ldots, x_n)\iff Q(x_1x,\ldots,x_nx)=kx^2.$$

In 1993, Conway and Schneeberger announced their celebrated \emph{Fifteen Theorem}, giving a criterion characterizing the universal positive-definite integer-matrix quadratic forms.   Specifically, they showed that any positive-definite integer-matrix form which represents the set of nine critical numbers $$\mathcal{S}_1=\{1,2,3,5,6,7,10,14,15\}$$ is universal (see \cite{Conway:universality,Bhargava:Fif}).  Kim, Kim, and Oh \cite{Kim:universal} presented an analogous criterion for $2$-universality, showing that a positive-definite integer-matrix lattice is $2$-universal if and only if it represents the set of forms $$\mathcal{S}_2=\left\{\bqf{1}{0}{0}{1},\bqf{2}{0}{0}{3},\bqf{3}{0}{0}{3},\bqf{2}{1}{1}{2},\bqf{2}{1}{1}{3},\bqf{2}{1}{1}{4}\right\}.$$   Oh \cite{Oh:MinRank} gave a similar criterion for $8$-universality, which we state in Theorem~\ref{8-crit} of Section~\ref{8-critSec}.

A set $\mathcal{S}$ of rank-$n$ lattices having the property that a lattice $L$ is $n$-universal if and only if $L$ represents every lattice in $\mathcal{S}$ is called an \textit{$n$-criterion set}.  Thus, for example, the set $\mathcal{S}_2$ obtained by Kim, Kim, and Oh \cite{Kim:universal} is a $2$-criterion set and the set $\mathcal{S}_1$ found by Conway \cite{Conway:universality} naturally gives the $1$-criterion set $$\left\{x^2,2x^2,3x^2,5x^2,6x^2,7x^2,10x^2,14x^2,15x^2\right\}.$$

The set $\mathcal{S}_1$ of the Fifteen Theorem is known to be unique (see \cite{Kim:finite}), in the sense that if $\mathcal{S}_1'$ is a set of integers such that a quadratic form is universal if and only if it represents the full set $\mathcal{S}_1'$, then $\mathcal{S}_1\subseteq \mathcal{S}_1'$. The author~\cite{Kom} recently obtained an analogous uniqueness result for the  $2$-criterion set $\mathcal{S}_2$.  

Kim, Kim, and Oh \cite{Kim:finite} have proven that $n$-criterion sets exist for all positive integers $n$.  However, the problems of finding and determining the uniquenesses of criterion sets have both proven to be difficult (see \cite{Kim:finite}).    Here, we advance both problems: We obtain the first characterization results for arbitrary $n$-criterion sets, from which we obtain the uniqueness of Oh's 8-universality criterion as a corollary.

\section{Notations and Terminology}
We use the lattice-theoretic language of quadratic form theory.  A complete introduction to this approach may be found in \cite{O'meara:Lattice}.

For a $\mathbb{Z}$-lattice (or hereafter, just \emph{lattice}) $L$ with basis $\{\mathbf{x}_1,\ldots,\mathbf{x}_n\}$, we write $L\cong \mathbb{Z}\mathbf{x}_{1}+\cdots+\mathbb{Z}\mathbf{x}_{n}$.  If $L$ is of the form $L= L_1\oplus L_2$ for sublattices $L_1$, $L_2$ of $L$ and $\inprod{L_1}{L_2}=0$ then we write $L\cong L_1\bot L_2$ and say that $L_1$ and $L_2$ are \textit{orthogonal}.  

 For a sublattice $\ell$ of $L_1\bot L_2$ which can be expressed in the form $$\ell\cong \mathbb{Z}(\mathbf{x}_{1,1}+\mathbf{x}_{2,1})+\cdots+\mathbb{Z}(\mathbf{x}_{1,n}+\mathbf{x}_{2,n})$$ with $\mathbf{x}_{i,j}\in L_i$, we denote $\ell(L_i):=\mathbb{Z}\mathbf{x}_{i,1}+\cdots+\mathbb{Z}\mathbf{x}_{i,n}$.  We naturally extend this notation to lattices $\ell$ represented by $L_1\bot L_2$.  We then say that a lattice is \emph{additively indecomposable} if either $\ell(L_1)\cong 0$ or $\ell(L_2)\cong 0$ whenever $L_1\bot L_2$ represents $\ell$.  Otherwise, we say that $\ell$ is \emph{additively decomposable}.

Finally, we use the lattice notation of Conway \cite{Conway:Low}.  In particular, $I_n$ is the rank-$n$ lattice of the form  $\left<1,\ldots,1\right>$ and $E_8$ is the unique even unimodular lattice of rank $8$.

\section{Characterization Results for $n$-criterion Sets}
In this section, we prove two results which partially characterize the contents of arbitrary $n$-criterion sets.
\begin{prop}\label{P1}
Any $n$-criterion set must include the lattice $I_n$.
\end{prop}
\begin{proof}
If $\mathcal{T}$ is a finite, nonempty set of rank-$n$ lattices not containing $I_n$, then every lattice $T\in \mathcal{T}$ may be written in the form $T\cong I_k\bot T'$, where $0\leq k<n$, the sublattice $T'$ is of rank $n-k$, and the first minimum of $T'$ is larger than $1$.  Indeed, any $I_k$-sublattice of  $T$ is unimodular and therefore splits $T$;  the condition on $T'$ follows from Minkowski reduction.

We may therefore write $\mathcal{T}$ in the form
$$\mathcal{T}=\bigcup_{k=0}^{n-1} \left\{I_k\bot T_{k,i}\right\}_{i=0}^{i_k},$$ where $0<|\mathcal{T}|=\sum_{k=0}^{n-1}i_k$ and each $T_{k,i}$ is a rank-$(n-k)$ lattice with first minimum greater than $1$.  Then, the lattice
$$I_{n-1}\bot \left(\left(\bot_{i=0}^{i_1}\ T_{0,i}\right)\bot\cdots\bot\left(\bot_{i=0}^{i_{n-1}}\ T_{n-1,i}\right)\right)$$ represents all of $\mathcal{T}$ but has does not represent $I_n$.  It follows that $\mathcal{T}$ is not an $n$-criterion set, so any $n$-criterion set must contain $I_n$.
\end{proof}

\begin{prop}\label{P2}
Let $\mathcal{E}$ be the set of additively indecomposable lattices of rank $n$.  If $|\mathcal{E}|>0$, then any $n$-criterion set must include at least one lattice $E\in \mathcal{E}$.
\end{prop}
\begin{proof}
If $\mathcal{T}=\{T_i\}_{i=1}^k$ is a finite, nonempty set of rank-$n$ lattices with $\mathcal{T}\cap \mathcal{E}=\emptyset$, then every lattice $T_i\in\mathcal{T}$ is additively decomposable.  It follows that the lattice 
$$ T_1\bot\cdots\bot T_k$$ represents all of $\mathcal{T}$ but does not represent any lattice in $\mathcal{E}$, since $T_1\bot\cdots\bot T_k$ has no rank-$n$ additively indecomposable sublattices.  Thus, $\mathcal{T}$ is not an $n$-criterion set.  It then follows that any $n$-criterion set must contain some lattice $E\in\mathcal{E}$.
\end{proof}

\begin{remark}It is clear that direct analogues of these two propositions hold in the more general setting of $\mathsf{S}$-universal lattices discussed in \cite{Kim:finite}.  In particular, suppose that $\mathsf{S}$ is an infinite set of lattices.  Then, if $n=\max\left\{k: I_k\in \mathsf{S}\right\}>0$, any finite set $\mathcal{S}_{\mathsf{S}}\subset \mathsf{S}$ with the property that a lattice $L$ represents every lattice $\ell\in \mathsf{S}$ if and only if $L$ represents every $\ell\in\mathcal{S}_{\mathsf{S}}$ must contain $I_n$.  Similarly, such a set $\mathcal{S}_{\mathsf{S}}$ must contain an additively indecomposable lattice if $\mathsf{S}$ does.
\end{remark}

\section{Uniqueness of the 8-criterion Set}
\label{8-critSec}
Oh obtained the following $8$-criterion set in \cite[remark on Theorem 3.1]{Oh:MinRank}:
\begin{Theorem}[Oh] \label{8-crit}The set $\mathcal{S}_8=\{I_8,E_8\}$ is an $8$-criterion set.\end{Theorem}

The set $\mathcal{S}_8$ is clearly a \textit{minimal} 8-criterion set, as for each $\ell\in \mathcal{S}_8$ there is a lattice which represents $\mathcal{S}_8\setminus \ell$ but does not represent $\ell$.  (The single lattice in $\mathcal{S}_8\setminus \ell$ suffices.)  Meanwhile, our characterization results imply the following corollary which strengthens Theorem~\ref{8-crit}:
\begin{cor}\label{unique}Every $8$-criterion set must contain $\mathcal{S}_8$ as a subset.
\end{cor}\begin{proof}
Since $E_8$ is the unique additively indecomposable lattice of rank $8$, the result follows directly from Propositions \ref{P1} and \ref{P2}.
\end{proof}  Corollary \ref{unique}, when combined with Theorem \ref{8-crit}, shows that $\mathcal{S}_8$ is the unique minimal 8-criterion set.

\section*{Acknowledgements}
The author is grateful to Pablo Azar, Noam D. Elkies, Andrea Hawksley, and Paul M. Kominers for their helpful comments and suggestions on the work and on earlier drafts of this article.


\begin{thebibliography}{KKO2}
\bibitem[Bh]{Bhargava:Fif}\textsc{M.\ Bhargava}: On the Conway-Schneeberger
  fifteen theorem, In: E.\ Bayer-Fluckiger, D.\ Lewis, A.\ Ranicki,
  eds: \textit{Contemp.\ Math.}\ \textbf{272},   Mathematical
  Association of America, Washington DC, (2000), 27--37.

\bibitem[C1]{Conway:Low}\textsc{J.\ H.\ Conway}: Low dimensional lattices. I. Quadratic forms of small determinant, \textit{Proc.\ Royal.\ Soc.\ Lond.\ A.} \textbf{418} (1988), 17--41.

\bibitem[C2]{Conway:universality} \textsc{J.\ H.\ Conway}: Universal quadratic forms and the
  fifteen theorem, In: E.\ Bayer-Fluckiger, D.\ Lewis, A.\ Ranicki,
  eds: \textit{Contemp.\ Math.} \textbf{272}   Mathematical
  Association of America, Washington DC,  (2000), 23--26.

\bibitem[KKO1]{Kim:finite}\textsc{B.\ M.\ Kim, M.-H.\ Kim, B.-K.\ Oh}:  A finiteness theorem for representability of integral quadratic forms by forms, \textit{J.\ Reine Angew.\ Math.}, \textbf{581}, (2005), 23--30.

\bibitem[KKO2]{Kim:universal} \textsc{B.\ M.\ Kim, M.-H.\ Kim, B.-K.\ Oh}: 2-universal positive definite integral quinary quadratic  forms, In: M.-H.\ Kim, J.\ S.\ Hsia, Y.\ Kitaoka, R.\ Schulze-Pillot,  eds: \textit{Contemp.\ Math.} \textbf{249}, Mathematical  Association of America, Washington DC, (1999), 51--62. 

\bibitem[Ko]{Kom}\textsc{S.\ D.\ Kominers}: Uniqueness of the 2-universality criterion, forthcoming, \textit{Note  Mat.} \textbf{28}(2), (2008).  (Preprint, \texttt{arXiv:0708.4340}.)

\bibitem[Oh]{Oh:MinRank} \textsc{B.-K.\ Oh}: Universal $\mathbb{Z}$-lattices of minimal rank, \textit{Proc.\ Amer.\ Math.\ Soc.} \textbf{128}(3), (1999), 683--689.

\bibitem[O'M]{O'meara:Lattice} \textsc{O.\ T.\ O'Meara}: \emph{Introduction to
  Quadratic Forms},  Springer-Verlag, New York, 1973.

\end{thebibliography}
\end{document}